\documentclass[11pt]{amsart}

\theoremstyle{plain}
\newtheorem{theorem}{Theorem}
\newtheorem{cor}{Corollary}
\newtheorem{lemma}{Lemma}
\newtheorem{proposition}{Proposition}
\newtheorem{fact}{Fact}
\newtheorem*{main1}{Transitivity Lemma}
\newtheorem*{main2}{Transitivity Lemma for weak minors}
\newtheorem*{main3}{Infinitary Transitivity}
\newtheorem*{main}{Claim}

\author{Miguel Couceiro}
\address{Department of Mathematics, Statistics and Philosophy\\
University of Tampere\\ 
Kalevantie 4, 33014 Tampere, Finland}
\email{Miguel.Couceiro@uta.fi}
\urladdr{http://mtl.uta.fi/miguel/}

\title[Generalized Functions and Relational Constraints]{Galois Connections for Generalized Functions and Relational Constraints}

\thanks{Partially supported by the Graduate school in Mathematical Logic MALJA.
Supported in part by grant $\sharp $28139 from the Academy of Finland}

\begin{document}

\begin{abstract} 
In this paper we focus on functions of the form $A^n\rightarrow \mathcal{P}(B)$, for possibly different
 arbitrary non-empty sets $A$ and $B$, and where $\mathcal{P}(B)$ denotes the set of all subsets of $B$.
These mappings are called \emph{multivalued functions}, and they 
generalize total and partial functions.
 We study Galois connections between these generalized functions and 
ordered pairs $(R,S)$ of relations
 on $A$ and $B$, respectively, called \emph{constraints}. We describe
the Galois closed sets, and decompose the associated Galois operators, 
by means of necessary and sufficient conditions which specialize, in the total
 single-valued case, to those given in \cite{CF}.  
\end{abstract}
\maketitle

\section{Introduction}

In \cite{Po1} and \cite{Po2}, P\"oschel developed a Galois theory for heterogeneous functions (i.e. functions on a 
cartesian product $A_{i_1}\times \ldots \times A_{i_n}$ to $A_j$, where the underlying sets belong to a family 
$(A_i)_{i\in I}$ of pairwise disjoint finite sets), in which the closed classes of functions are defined 
by invariant multisorted relations $R=\cup _{i\in I}R_i$ where $R_i\subseteq A_{i}^m$, and dually, 
the closed systems of relations are charaterized by the functions preserving them (for further background,
 see also \cite{PK}).
Still in the finite case, Pippenger studied in \cite{Pi}, the particular bi-sorted case of finite functions 
(i.e. mappings of the form $f:A^n\rightarrow B$), and introduced a Galois framework
 in which the dual objects are replaced by ordered pairs $(R,S)$ of relations
 on $A$ and $B$, respectively, called \emph{constraints},
and where the multisorted preservation is replaced by the more stringent notion of 
constraint satisfaction. This latter theory was extended in \cite{CF}
 by removing the finiteness condition on the underlying sets $A$ and $B$.

In this paper we study the more general notion of
\emph{multivalued functions}, that is, mappings of the form $A^n\rightarrow \mathcal{P}(B)$,
 where $\mathcal{P}(B)$ denotes the set of all subsets of $B$. 
We introduce the Galois connection between sets of these generalized functions and sets of 
 constraints $(R,S)$ (where $R\subseteq A^m$ and $S\subseteq B^m$), based on a 
more general notion of constraint satisfaction (see Section 2). 
 Although the functions that we consider can still be treated as maps $A^n\rightarrow C$, where $C=\mathcal{P}(B)$,
 our approach extends the framework in \cite{Pi} and \cite{CF}, because we take as dual objects constraints 
  in which the ``consequent" $S$ is a relation defined over $B$,
 and not over $C=\mathcal{P}(B)$ as it is defined in these papers.

We describe the Galois closed classes of 
multivalued functions (Section 3) and the Galois closed sets of constraints (Section 4), in terms of
 closures which essentially extend to the multivalued case the conditions presented in \cite{CF}.
We consider further Galois connections by
restricting the set of primal objects to partial functions,
and to total multivalued functions, i.e.\@ mappings $A^n\rightarrow \mathcal{P}(B)$
which are non-empty-valued on every $n$-tuple over $A$.
 (For universal algebraic analogues, see e.g. \cite{FR} and \cite{Bo},
 respectively, and \cite{Ro} for an unified approach
 to these extensions.) As corollaries we obtain the characterizations, given in \cite{CF},
of the closed classes of single-valued functions (see Corollary 1 (c)),
 and the corresponding dual closed sets of 
constraints (see Corollary 3).  
Furthermore, we present factorizations of the closure maps associated with the above-mentioned Galois
connections, as compositions of simpler operators.

\section{Basic notions}

Throughout the paper, we shall always consider arbitrary non-empty base sets $A$, $B$, etc. Also, the integers 
$n$, $m$, etc., are assumed to be positive and thought of as Von Neumann ordinals, i.e. each ordinal is the non-empty set
of lesser ordinals. With this formalism, $n$-tuples over a set $A$ are just unary maps 
from $n=\{0,\ldots ,n-1\}$ to $A$. Thus an \emph{$m$-ary relation $R$ on $A$} 
(i.e. a subset $R\subseteq A^m$) is viewed as a set of unary maps ${\bf a}=(a_i\mid i\in m)$ from $m$ to $A$.  
Furthermore, we shall distinguish between empty relations of different arities, and we write $\emptyset ^m$
to denote the $m$-ary empty relation. For $m=1$, we use $\emptyset $ (instead of $\emptyset ^1$) to denote
the unary empty relation. 
In order to present certain concepts in a unifying setting, e.g. those of total multivalued and 
partial functions, we shall think of functions as having specific domain, codomain and graph.

An $n$-ary \emph{multivalued function on $A$ to $B$} is 
a map $f: A^n \rightarrow \mathcal{P}(B)$, where $\mathcal{P}(B)$ denotes the set of all subsets of $B$.
For $A=B$, these functions are called \emph{multioperations} or \emph{multifunctions on} $B$,
 and for $A=\mathcal{P}(B)$ the maps $f: \mathcal{P}(B)^n \rightarrow \mathcal{P}(B)$ are said to be \emph{lifted} 
(see \cite{DP}). By a \emph{class of multivalued functions} we simply mean a set of multivalued functions
 of various arities. If $f: A^n \rightarrow \mathcal{P}(B)$ is 
non-empty-valued on every $n$-tuple over $A$,
 then $f$ is said to be a \emph{total multivalued function on $A$ to $B$}. These indeed correspond to 
total functions in the usual sense, i.e. to each $n$-tuple over $A$, they associate at least one element of $B$.
We denote by $\Theta _{AB}$ the class of all multivalued functions on $A$ to $B$,
and by $\Theta _{AB}^t$ the class of all total multivalued functions on $A$ to $B$.

In this paper we also consider the following particular cases of multivalued functions. 
We say that a multivalued function $f: A^n \rightarrow \mathcal{P}(B)$ is a 
\emph{partial function on $A$ to $B$} if it is either empty or singleton-valued on every $n$-tuple over $A$,
i.e. if for every ${\bf a}$ in $A^n$, we have $f({\bf a})=\emptyset $ or $f({\bf a})=\{b\}$, 
for some $b$ in $B$. Although partial functions on $A$ to $B$ are usually defined as maps 
$p: D \rightarrow B$ where $D\subseteq A^n$ 
(see e.g. \cite{BW}, and for partial operations, see e.g. \cite{R, BHP}), it is easy to establish  
 a complete correspondence between these definitions.
For each positive integer $n$, 
the $n$-ary partial function $e_n$ which has empty value on every element of $A^n$, is called
the $n$-ary \emph{empty-valued function}. With $\Theta _{AB}^p$ we denote the class of all partial
 functions on $A$ to $B$.

Observe that the functions of several variables on $A$ to $B$ considered in [CF], 
correspond to the partial functions on $A$ to $B$ (as formerly defined) which are, in addition, total.
In other words, there is a bijection between $\Theta _{AB}^s=\Theta _{AB}^t\cap \Theta _{AB}^p$ and 
$\cup _{n\geq 1}B^{A^n}$. In this paper we shall refer to functions in $\Theta _{AB}^s$
as \emph{single-valued functions on $A$ to $B$}.

For a multivalued function $f: A^n \rightarrow \mathcal{P}(B)$ 
and $m$-tuples ${\bf a}^1,\ldots ,{\bf a}^n$ over $A$, we write $f({\bf a}^1\ldots {\bf a}^n)$
for the $m$-ary relation on $B$, defined by
 \begin{displaymath}
 f({\bf a}^1\ldots {\bf a}^n)= {\Pi }_{i\in m}f(({\bf a}^1\ldots {\bf a}^n)(i)) 
\end{displaymath}
where $({\bf a}^1\ldots {\bf a}^n)(i)=({\bf a}^1(i)\ldots {\bf a}^n(i)) $. 
Note that if $f(({\bf a}^1\ldots {\bf a}^n)(i))=\emptyset $, for some $i\in m$, then 
$f({\bf a}^1\ldots {\bf a}^n)=\emptyset ^m$.
If $R$ is an $m$-ary relation on $A$, we denote by $fR$ the $m$-ary relation on $B$, defined by
 \begin{displaymath}
 fR=\cup \{f({\bf a}^1\ldots {\bf a}^n):{\bf a}^1,\ldots ,{\bf a}^n \in R\}. 
\end{displaymath}

An $m$-ary \emph{$A$-to-$B$ relational constraint} is an ordered pair $(R,S)$ of relations  
$R\subseteq A^m$ and $S\subseteq B^m$, called \emph{antecedent} and \emph{consequent}, respectively,
of the constraint. A multivalued function $f: A^n \rightarrow \mathcal{P}(B)$ 
is said to \emph{satisfy} the constraint $(R,S)$ 
if $fR\subseteq S$. 
Observe that for each $1\leq m$, every multivalued function on $A$ to $B$ satisfies the $m$-ary 
\emph{empty constraint} $(\emptyset ^m,\emptyset ^m)$,
 and the $m$-ary \emph{trivial constraint} $(A^m,B^m)$. Moreover,  
every partial function on $A$ to $B$ satisfies the \emph{binary equality constraint} $(=_A,=_B)$,
where $=_A$ and $=_B$ denote the equality relations on $A$ and on $B$, respectively.

For a set $\mathcal{T}$ of $A$-to-$B$ constraints, we denote by ${\bf mFSC}(\mathcal{T})$ 
 the class of all multivalued functions on $A$ to $B$ satisfying every constraint in $\mathcal{T}$.
The notation ${\bf mFSC}$ stands for ``multivalued functions satisfying constraints". 
 A class $\mathcal{M}$ of multivalued functions on $A$ to $B$
 is said to be \emph{definable} by a set $\mathcal{T}$ of $A$-to-$B$ constraints,
 if $\mathcal{M}={\bf mFSC}(\mathcal{T})$.
Similarly, the classes of 
\begin{itemize}
\item[(i)] total multivalued functions of the form 
${\bf tFSC}(\mathcal{T})=\Theta _{AB}^t\cap {\bf mFSC}(\mathcal{T})$,
\item[(ii)] partial functions of the form 
${\bf pFSC}(\mathcal{T})=\Theta _{AB}^p\cap {\bf mFSC}(\mathcal{T})$, and 
\item[(iii)] single-valued functions of the form 
${\bf sFSC}(\mathcal{T})=\Theta _{AB}^s\cap {\bf mFSC}(\mathcal{T})$ 
\end{itemize}
are said to be 
\emph{definable within} $\Theta _{AB}^t,\Theta _{AB}^p$, and $\Theta _{AB}^s$, respectively,
 by the set $\mathcal{T}$. 

Dually, for a class $\mathcal{M}$ of multivalued functions on $A$ to $B$, we denote by 
${\bf CSF}(\mathcal{M})$
 the set of all $A$-to-$B$ constraints satisfied by every function in $\mathcal{M}$. 
Note that ${\bf CSF}$ stands for ``constraints satisfied by functions". 
In analogy with the function case, a set $\mathcal{T}$ of $A$-to-$B$ constraints
is said to be \emph{characterized} by a set $\mathcal{M}$ of multivalued functions,
if $\mathcal{T}={\bf CSF}(\mathcal{M})$.

Let $V$ and $W$ be arbitrary sets. It is well known that each binary relation 
$\triangleright \subseteq V\times W$
induces a \emph{Galois connection} between $V$ and $W$, determined by the pair of mappings
$v:\mathcal{P}(V)\rightarrow \mathcal{P}(W)$ and $w:\mathcal{P}(W)\rightarrow \mathcal{P}(V)$, defined
as follows:
\begin{displaymath}
v(X)=\{b\in W: a\triangleright b,\textrm{ for every } a\in X \}
 \end{displaymath}
\begin{displaymath}
w(Y)=\{a\in V: a\triangleright b,\textrm{ for every } b\in Y \}. 
\end{displaymath}
The associated operators $X\mapsto (w\circ v)(X)$ and $Y\mapsto (v\circ w)(Y)$ are
\emph{extensive}, \emph{monotone} and \emph{idempotent}, i.e. they satisfy the following conditions
 \begin{itemize}
\item[E. \quad ] for every $X\subseteq V$  and  $Y\subseteq W$,

$X\subseteq (w\circ v)(X)$  and  $Y\subseteq (v\circ w)(Y)$,
\item[M. \quad ] if $X'\subseteq X$  and  $Y'\subseteq Y$, then

$(w\circ v)(X')\subseteq (w\circ v)(X)$  and  $(v\circ w)(Y')\subseteq (v\circ w)(Y)$, 
\item[I. \quad ] for every $X\subseteq W$  and  $Y\subseteq V$,

$(w\circ v)((w\circ v)(X))= (w\circ v)(X)$  and  $(v\circ w)((v\circ w)(Y))= (v\circ w)(Y)$),
\end{itemize}
  respectively. In other words, $w\circ v$ and $v\circ w$ are \emph{closure operators} on $V$ and $W$,
 respectively, and the sets $X$ and $Y$ satisfying $(w\circ v)(X)=X$ and $(v\circ w)(Y)=Y$ are the 
(\emph{Galois}) \emph{closed sets} associated with $v$ and $w$. Moreover, $(w\circ v)(X)$ and 
$(v\circ w)(Y)$ are the smallest closed sets containing
 $X\subseteq V$ and $Y\subseteq W$, respectively, and are said to be \emph{generated} by $X$ and $Y$.
(For background on Galois connections see e.g. \cite{O}, and \cite{Pi2} for a later reference.)

Based on the relation of constraint satisfaction (between multivalued functions and constraints),
 we define the Galois connection ${\bf mFSC}-{\bf CSF}$
between sets of multivalued functions and sets of relational constraints.
Let $V$ be the class of all multivalued functions on $A$ to $B$,
 and $W$ the set of all $A$-to-$B$ relational constraints. 
Interpreting $\triangleright $ as the binary relation ``satisfies", we have that:
\begin{itemize}
\item[(a)] $v(\mathcal{K})={\bf CSF}(\mathcal{K})$
for every $\mathcal{K}\subseteq V$, and
\item[(b)] $w(\mathcal{T})={\bf mFSC}(\mathcal{T})$
for every $\mathcal{T}\subseteq W$.
\end{itemize}
Similarly, we define the correspondences ${\bf tFSC}-{\bf CSF}$,
${\bf pFSC}-{\bf CSF}$, and
${\bf sFSC}-{\bf CSF}$, by restricting $V$ to $\Theta _{AB}^t,\Theta _{AB}^p$,
 and $\Theta _{AB}^s$, respectively.

With this terminology, the classes of generalized functions definable by constraints are exactly the 
closed sets of functions associated with the corresponding Galois connections, and 
the sets of relational constraints characterized by generalized functions correspond to the 
dual Galois closed sets.

\section{Galois closed Sets of Generalized Functions}

We say that an $n$-ary multivalued function $g$ on $A$ to $B$ is a \emph{value restriction} of an 
$n$-ary multivalued function $f$ on $A$ to $B$, if for every ${\bf a}\in A^n$ we have 
$g({\bf a})\subseteq f({\bf a})$. A class $\mathcal{M}$ of multivalued functions on $A$ to $B$
is said to be \emph{closed under taking value restrictions}
 if every value restriction of a member of $\mathcal{M}$ is also in $\mathcal{M}$.
(In \cite{Bo}, where $A=B$ is finite,
 the non-empty value restrictions of a total
multivalued function $f$ are called \emph{subfunctions} of $f$.)

We now introduce a key concept which extends that of \emph{simple variable substitution} 
(appearing in \cite{CF}, and referred to as \emph{minor} in \cite{Pi}) to multivalued functions,
 and subsumes value restrictions. 
We say that an $m$-ary multivalued function $g$ from $A$ to $B$
is obtained from an $n$-ary multivalued function $f$ from $A$ to $B$
 by \emph{restrictive variable substitution},
if there is a map $l$ from $n$ to $m$ such that 
\begin{displaymath}
g({\bf a})\subseteq f({\bf a}\circ l)
\end{displaymath}
for every $m$-tuple ${\bf a}\in A^m$. If $g$ is non-empty valued, i.e. $g({\bf a})\not=\emptyset $ 
for every ${\bf a}\in A^m$, then we say that $g$ is obtained from $f$
  by \emph{non-empty restrictive variable substitution}. Note that within $\Theta _{AB}^s$,
the inclusion may be replaced by equality, and in this case we use the term "simple" instead of "restrictive"
(see \cite{CF}).
 
 A class $\mathcal{M}$ of multivalued functions of several variables is said to be 
\emph{closed under restrictive variable substitutions}
if every multivalued function obtained from a function $f$ in $\mathcal{M}$
 by restrictive variable substitution
 is also in $\mathcal{M}$.  
 For any class $\mathcal{M}$ of multivalued functions, we denote by
 ${\bf RVS}(\mathcal{M})$ the smallest class containing $\mathcal{M}$, and
 closed under ``restrictive variable substitutions".
Similarly, we use ${\bf RVS}^t(\mathcal{M})$ to denote the smallest class containing $\mathcal{M}$, and
 closed under non-empty restrictive variable substitutions.
 By the definitions above it follows:

\begin{fact} 
For any class $\mathcal{M}\subseteq \Theta _{AB}$, we have
\begin{itemize}
\item[(i)] ${\bf RVS}^t(\Theta _{AB}^t\cap \mathcal{M})=\Theta _{AB}^t\cap {\bf RVS}(\mathcal{M})$, 
\item[(ii)] ${\bf RVS}(\Theta _{AB}^p\cap \mathcal{M})\subseteq \Theta _{AB}^p$, and
\item[(iii)] ${\bf RVS}^t(\Theta _{AB}^s\cap \mathcal{M})\subseteq \Theta _{AB}^s$.   
\end{itemize} 
\end{fact} 
It is easy to check that every member of ${\bf RVS}(\mathcal{M})$, and thus of ${\bf RVS}^t(\mathcal{M})$, 
satisfies every constraint in ${\bf CSF}(\mathcal{M})$.

Due to the fact that we consider relational constraints of finite arities, the non-satisfaction of 
a constraint by a multivalued function is always detected in a finite restriction to the domain of the function.
 For this reason, we recall the the concept of ``local closure".

A class $\mathcal{M}\subseteq \Theta _{AB}$ is said 
to be \emph{locally closed} 
if it contains every multivalued function $f: A^n \rightarrow \mathcal{P}(B)$ for which every restriction 
to a finite subset of its domain $A^n$ coincides with a restriction of some member of $\mathcal{M}$.
Obviously, if $A$ is finite, then every class $\mathcal{M}\subseteq \Theta _{AB}$ is locally closed.

It is not difficult to verify that this property is indeed a necessary condition on classes definable by constraints.
But even if closure under restrictive variable substitutions is assumed, say on a class
$\mathcal{M}\subseteq \Theta _{AB}$, it is not sufficient to guarantee the existence of a set of constraints defining   
$\mathcal{M}$. 

To illustrate, let $A=B=\{0,1\}$, 
 and let $\mathcal{M}$ be the class containing only the unary constant function ${\bf 0}:x\mapsto \{0\}$,
 the unary constant function ${\bf 1}:x\mapsto \{1\}$,
 and the unary ``identity" ${\bf i}:x\mapsto \{x\}$, for every $x\in A=B$.
Consider the unary multivalued functions $f:A \rightarrow \mathcal{P}(B)$  
defined by
 \begin{displaymath}
 f(x)= {\bf 1}(x) {\cup } {\bf i}(x) \textrm{ i.e. } f(0)=B=\{0,1\}\textrm{ and } f(1)= \{1\},
 \end{displaymath}
and $g:A \rightarrow \mathcal{P}(B)$ defined by 
 \begin{displaymath}
 g(x)={\bf 0}(x) {\cup } {\bf 1}(x) \textrm{ i.e. } g(x)=B=\{0,1\}, \textrm{ for all $x\in A=\{0,1\}$.} 
 \end{displaymath}
Note that 
 \begin{displaymath}
{\Pi }_{{\bf a}\in A}f({\bf a})\subseteq {\cup }_{h\in \mathcal{M}}{\Pi }_{{\bf a}\in A}h({\bf a})
\subset   {\Pi }_{{\bf a}\in A}g({\bf a})=B^2
 \end{displaymath} 
Thus every constraint satisfied by every function in $\mathcal{M}$
 must be also satisfied by the function $f$, but there are constraints satisfied by every function in $\mathcal{M}$
 which are not satisfied by $g$. 
 
Clearly, ${\bf RVS}(\mathcal{M})$ is locally closed
and closed under restrictive variable substitutions. 
 Also, it is not difficult to check that $f$ and $g$ do not belong to
${\bf RVS}(\mathcal{M})$. 
By the fact that $f$ satisfies every constraint satisfied by the members ${\bf RVS}(\mathcal{M})$, it follows
that ${\bf RVS}(\mathcal{M})$ is properly contained in every definable class containing $\mathcal{M}$.
Furthermore, from the fact that $g$ does not satisfy every constraint in ${\bf CSF}(\mathcal{M})$, 
we conclude that a class definable by constraints 
does not necessarily contain all functions which are defined as the ``union" of a family of members of the class.

This example motivates the introduction of the following concept which extends local closure.
We say that a class $\mathcal{M}$ of multivalued functions on $A$ to $B$ is 
\emph{closed under local coverings}
 if it contains every multivalued function $f$ 
 on $A$ to $B$ such that for every finite subset $F\subseteq A^n$, there is
a non-empty family $(f_i)_{i\in I}$ of members of $\mathcal{M}$ of the same arity as $f$, such that
 \begin{displaymath}
{\Pi }_{{\bf a}\in F}f({\bf a})\subseteq {\cup }_{i\in I}{\Pi }_{{\bf a}\in F}f_i({\bf a})  \qquad (1)
 \end{displaymath} 
Clearly, if a class is closed under local coverings, then it is locally closed. Moreover, within $\Theta _{AB}^p$,
the families $(f_i)_{i\in I}$ above, all reduce to singleton families, and within $\Theta _{AB}^s$,
 the inclusion relation can be replaced by equality, 
i.e. closure under local coverings coincides with local closure. 

Note also that condition $(1)$ is equivalent to
\begin{displaymath}
{\Pi }_{{\bf a}\in F}f({\bf a})\subseteq {\cup }_{g\in \mathcal{M}_n}{\Pi }_{{\bf a}\in F}g({\bf a}).
 \end{displaymath}
 where $n$ denotes the arity of $f$, and $\mathcal{M}_n$ 
is the set of all $n$-ary multivalued functions in $\mathcal{M}$. 
    
The smallest class of multivalued functions containing $\mathcal{M}$, and closed under ``local coverings" 
 is denoted by ${\bf LC}(\mathcal{M})$. It is not difficult to see that
${\bf LC}(\mathcal{M})$ is the class of functions 
obtained from $\mathcal{M}$ by adding all those functions whose restriction to each finite subset 
of its domain is contained in that restriction of some union of members of $\mathcal{M}$.
Moreover, we define:
\begin{itemize}
\item[(i)] ${\bf pLC}(\mathcal{M})=\Theta _{AB}^p\cap {\bf LC}(\mathcal{M})$, 
\item[(ii)] ${\bf tLC}(\mathcal{M})=\Theta _{AB}^t\cap {\bf LC}(\mathcal{M})$, and
\item[(iii)] ${\bf sLC}(\mathcal{M})=\Theta _{AB}^s\cap {\bf LC}(\mathcal{M})$,   
\end{itemize} 
 and we say that a class $\mathcal{M}$ is \emph{closed under partial local coverings},
 \emph{closed under total local coverings}, or 
\emph{closed under simple local coverings}, if ${\bf pLC}(\mathcal{M})=\mathcal{M}$, 
${\bf tLC}(\mathcal{M})=\mathcal{M}$, or ${\bf sLC}(\mathcal{M})=\mathcal{M}$, respectively.

\begin{proposition} 
Consider arbitrary non-empty sets $A$ and $B$, and let $\mathcal{M}$ be a class of multivalued functions.
\begin{itemize}
\item[(i)] The operators $\mathcal{M}\mapsto {\bf RVS}(\mathcal{M})$ and 
$\mathcal{M}\mapsto {\bf RVS}^t(\mathcal{M})$ are closure operators on $\Theta _{AB}$ and
 $\Theta _{AB}^t$, respectively. Moreover,
 they are also closure operators on $\Theta _{AB}^p$ and $\Theta _{AB}^s$, respectively.
\item[(ii)] The operators $\mathcal{M}\mapsto {\bf LC}(\mathcal{M})$,
$\mathcal{M}\mapsto {\bf tLC}(\mathcal{M})$, $\mathcal{M}\mapsto {\bf pLC}(\mathcal{M})$, and
$\mathcal{M}\mapsto {\bf sLC}(\mathcal{M})$ are closure operators
 on $\Theta _{AB}$, $\Theta _{AB}^t$, $\Theta _{AB}^p$ and $\Theta _{AB}^s$, respectively.
\item[(iii)] If ${\bf RVS}(\mathcal{M})=\mathcal{M}$,
 then ${\bf RVS}({\bf LC}(\mathcal{M}))={\bf LC}(\mathcal{M})$.
\item[(iv)] If ${\bf RVS}^t(\mathcal{M})=\mathcal{M}$,
 then ${\bf RVS}^t({\bf tLC}(\mathcal{M}))={\bf tLC}(\mathcal{M})$.
\end{itemize}
\end{proposition}

\begin{proof}
Statements $(i)$ and $(ii)$ follow immediately from the above definitions and Fact 1.
The proof of $(iii)$ is analogous to that of Proposition 1 (a) in \cite{Co}. 
We show that ${\bf RVS}({\bf LC}(\mathcal{M}))\subseteq {\bf LC}(\mathcal{M})$. 
      
 Suppose that $g\in {\bf RVS}({\bf LC}(\mathcal{M}))$, say of arity $m$. Thus, 
there is an $n$-ary function $f$ in ${\bf LC}(\mathcal{M})$, 
and a map $l:n\rightarrow m$, such that 
\begin{displaymath}
g({\bf a})\subseteq f({\bf a}\circ l)
\end{displaymath}
for every $m$-tuple ${\bf a}\in A^m$. 
Let $F$ be a finite subset of $A^m$. 
We show that there is a non-empty 
family $(g^{F}_i)_{i\in I}$ of $m$-ary members of $\mathcal{M}$, such that
 \begin{displaymath}
{\Pi }_{{\bf a}\in F}g({\bf a})\subseteq {\cup }_{i\in I}{\Pi }_{{\bf a}\in F}g^{F}_i({\bf a}).
 \end{displaymath}
Consider the finite subset $F'\subseteq A^n$, defined by
\begin{displaymath}
F'=\{{\bf a}\circ l: {\bf a}\in F\} 
\end{displaymath}
From the fact that $f\in {\bf LC}(\mathcal{M})$, it follows that there is a non-empty 
family $(f^{F'}_i)_{i\in I}$ of $m$-ary members of $\mathcal{M}$, such that
 \begin{displaymath}
{\Pi }_{{\bf a}'\in F'}f({\bf a}')\subseteq {\cup }_{i\in I}{\Pi }_{{\bf a}'\in F'}f^{F'}_i({\bf a}').
 \end{displaymath}
For each $i\in I$, let $g^{F}_i$ be the $m$-ary function defined by 
\begin{displaymath}
g^{F}_i({\bf a})=f^{F'}_i({\bf a}\circ l)
\end{displaymath}
for every $m$-tuple ${\bf a}\in A^m$. Note that $(g^{F}_i)_{i\in I}$ is a family of members of
 $\mathcal{M}$, because
${\bf RVS}(\mathcal{M})=\mathcal{M}$. 
By the definition of $(f^{F'}_i)_{i\in I}$ and $(g^{F}_i)_{i\in I}$, it follows that,
 for every $m$-tuple ${\bf a}\in F$, 
\begin{displaymath}
{\Pi }_{{\bf a}\in F}g({\bf a})\subseteq {\Pi }_{{\bf a}\in F}f({\bf a}\circ l)\subseteq 
{\cup }_{i\in I}{\Pi }_{{\bf a}\in F}f^{F'}_i({\bf a}\circ l)={\cup }_{i\in I}{\Pi }_{{\bf a}\in F}g^{F}_i({\bf a}).
\end{displaymath}
Since the above argument works for every finite subset $F$ of $A^m$, we have that $g$
 is in ${\bf LC}(\mathcal{M})$.
The proof of $(iv)$ can be obtained by proceeding in analogy with the proof of $(iii)$.
\end{proof}

Using $(i)$, $(ii)$ and $(iii)$ of Proposition 1, it is straightfoward to check that, for every class
$\mathcal{M}\subseteq \Theta _{AB}$,  
${\bf LC}({\bf RVS}(\mathcal{M}))$ is the smallest class containing $\mathcal{M}$,
which is closed under local coverings, and closed under restrictive variable substitutions. Similarly, using  
$(i)$, $(ii)$ and $(iv)$ of Proposition 1, it is easy to check that, for every class
$\mathcal{M}\subseteq \Theta _{AB}^t$,  
${\bf tLC}({\bf RVS}^t(\mathcal{M}))$ is the smallest class containing $\mathcal{M}$,
which is closed under total local coverings, and closed under non-empty restrictive variable substitutions.

Our first main result provides necessary and sufficient conditions for a 
class of multivalued functions to be definable by relational constraints:

\begin{theorem}
Consider arbitrary non-empty sets $A$ and $B$.
 For any class $\mathcal{M}$ of multivalued functions on $A$ to $B$,
the following conditions are equivalent:
\begin{itemize}
\item[(i)] $\mathcal{M}$ is closed under local coverings, contains the unary empty-valued function $e_1$,
and is closed under restrictive variable substitutions;
\item[(ii)] $\mathcal{M}$ is definable by some set of $A$-to-$B$ constraints.
\end{itemize}
\end{theorem}

\begin{proof}
First, we prove $(ii)\Rightarrow (i)$. Clearly, the unary empty-valued function satisfies every constraint
and it is easy to see that if a multivalued function $f$
 satisfies a constraint $(R,S)$, then every
function obtained from $f$ by restrictive variable substitution also satisfies $(R,S)$. Therefore,
any function class $\mathcal{M}$ definable by a set of constraints contains
 the unary empty-valued function,
 and is closed under restrictive variable substitutions.

To see that $\mathcal{M}$ is closed under local coverings, consider an $n$-ary multivalued function
 $f\not\in \mathcal{M}$. From $(ii)$ it follows that 
there is an $m$-ary constraint $(R,S)$ which is not satisfied by $f$ but
  satisfied by every function $g$ in $\mathcal{M}$. Hence, for some ${\bf a}^1,\ldots ,{\bf a}^n \in R$,
we have $f({\bf a}^1\ldots {\bf a}^n) \not \subseteq S$, and 
$g({\bf a}^1\ldots {\bf a}^n)\subseteq S$ for every $g\in \mathcal{M}_n$, where $\mathcal{M}_n$ 
is the set of all $n$-ary multivalued functions in $\mathcal{M}$.
 Thus, 
 \begin{displaymath}
{\Pi }_{i\in m}f(({\bf a}^1\ldots {\bf a}^n)(i))\not\subseteq {\cup }_{g\in \mathcal{M}_n}{\Pi }_{i\in m} g(({\bf a}^1\ldots {\bf a}^n)(i)).   
 \end{displaymath}

To prove the implication $(i)\Rightarrow (ii)$, assume $(i)$. 
We proceed as in the proof of Theorem 1 in \cite{CF},
 and show that for every function $f\not\in \mathcal{M}$, there is a constraint $(R_f,S_f)$
 satisfied by every member of $\mathcal{M}$, but not satisfied by $f$. This suffices to conclude $(ii)$
 because $\mathcal{M}={\bf mFSC}(\{(R_f,S_f):f\not\in \mathcal{M}\})$, i.e.
 the set $\{(R_f,S_f):f\not\in \mathcal{M}\}$ defines the class $\mathcal{M}$.

So suppose that $f\not\in \mathcal{M}$, say of arity $n$. 
Since $\mathcal{M}$ is closed under local coverings, 
 there is a finite subset $F\subseteq A^n$
 such that
  \begin{displaymath}
{\Pi }_{{\bf a}\in F}f({\bf a})\not\subseteq {\cup }_{i\in I}{\Pi }_{{\bf a}\in F}f_i({\bf a})
 \end{displaymath}
for every non-empty family $(f_i)_{i\in I}$ of $n$-ary members of $\mathcal{M}$.
In particular, 
\begin{displaymath} 
{\Pi }_{{\bf a}\in F}f({\bf a})\not\subseteq {\cup }_{g\in \mathcal{M}_n}{\Pi }_{{\bf a}\in F}g({\bf a})
 \end{displaymath}
where $\mathcal{M}_n$ is the set of all $n$-ary multivalued functions in $\mathcal{M}$. 
 Observe that $F$ can not be empty, and that $f$ can not be empty-valued on $F$. 
Let ${\bf a}^1,\ldots ,{\bf a}^n$ be tuples in $A^{\mid F\mid }$ such that
$F=\{({\bf a}^1\ldots {\bf a}^n)(i):i\in \mid  F\mid \}$, 
and let $(R,S)$ be the constraint whose antecedent is $R=\{{\bf a}^1,\ldots ,{\bf a}^n\}$, and whose
 consequent is defined by $S=\cup _{g\in \mathcal{M}_n}g({\bf a}^1\ldots {\bf a}^n)$.
Clearly, $f$ does not satisfy the $A$-to-$B$ constraint $(R,S)$, and since $\mathcal{M}$ is closed
under restrictive variable substitutions, it follows that 
every function in $\mathcal{M}$ satisfies $(R,S)$.
Thus for every function $f\not\in \mathcal{M}$, there is a constraint $(R_f,S_f)$ satisfied by every member of 
$\mathcal{M}$, but not satisfied by $f$.   
\end{proof}

From Fact 1 and Proposition 1 $(i)$, we obtain as particular cases of Theorem 1, characterizations
for classes of multivalued functions of the form ${\bf pFSC}(\mathcal{T})$, ${\bf tFSC}(\mathcal{T})$, and
${\bf sFSC}(\mathcal{T})$:

\begin{cor}
Consider arbitrary non-empty sets $A$ and $B$.
\begin{itemize}
\item[(a)] A class $\mathcal{M}^p$ of partial functions is 
definable within $\Theta _{AB}^p$ by some set of $A$-to-$B$ constraints if and only if it
is closed under partial local coverings, contains the unary empty-valued function,
and is closed under restrictive variable substitutions.
\item[(b)] A class $\mathcal{M}^t$ of total multivalued functions is 
definable within $\Theta _{AB}^t$ by some set of $A$-to-$B$ constraints if and only if it
is closed under total local coverings, and is closed under non-empty restrictive variable substitutions.
\item[(c)] \emph{(In [CF]:)} A class $\mathcal{M}^s$ of single-valued functions is 
definable within $\Theta _{AB}^s$ by some set of $A$-to-$B$ constraints if and only if it
is closed under simple local coverings, and is closed under simple variable substitutions.
\end{itemize}
\end{cor}

We finish this section with the factorizations of the Galois closure operators on
$\Theta _{AB}$, $\Theta _{AB}^p$, $\Theta _{AB}^t$, and $\Theta _{AB}^s$, as compositions 
of the operators induced by the above closure conditions:

\begin{proposition}
Consider arbitrary non-empty sets $A$ and $B$.
 For any class of multivalued functions $\mathcal{M}\subseteq \Theta _{AB}$, the following hold:
\begin{itemize}
\item[(i)] ${\bf mFSC}({\bf CSF}(\mathcal{M}))={\bf LC}({\bf RVS}(\mathcal{M}\cup \{e_1\}))$. 
\item[(ii)] If $\mathcal{M}\subseteq \Theta _{AB}^p$, then 
${\bf pFSC}({\bf CSF}(\mathcal{M}))={\bf pLC}({\bf RVS}(\mathcal{M}\cup \{e_1\}))$. 
\item[(iii)] If $\mathcal{M}\subseteq \Theta _{AB}^t$, then 
${\bf tFSC} ({\bf CSF}(\mathcal{M}))={\bf tLC}({\bf RVS}^t(\mathcal{M}))$.
\item[(iv)] If $\mathcal{M}\subseteq \Theta _{AB}^s$, then 
${\bf sFSC}({\bf CSF}(\mathcal{M}))={\bf sLC}({\bf RVS}^t(\mathcal{M}))$. 
\end{itemize}
\end{proposition}

\begin{proof}
To see that $(i)$ holds, note first that ${\bf mFSC}({\bf CSF}(\mathcal{M}))$
is the smallest Galois closed set of multivalued functions containing $\mathcal{M}$. 
Thus it follows from Theorem 1 that ${\bf mFSC}({\bf CSF}(\mathcal{M}))$ is the 
is the smallest class containing $\mathcal{M}\cup \{e_1\}$,
which is closed under local coverings and closed under restrictive variable substitutions.
By the comments following the proof of Proposition 1, we get 
${\bf mFSC}({\bf CSF}(\mathcal{M}))={\bf LC}({\bf RVS}(\mathcal{M}\cup \{e_1\}))$.
In other words, $(i)$ holds. 
The proof of $(ii)$, $(iii)$ and $(iv)$ can be obtained similarly, using  
Corollary 1 and Proposition 1. 
\end{proof}

\section{Galois closed Sets of Relational Constraints}

In order to describe the dual closed sets, i.e. the sets of constraints characterized by multivalued functions, 
we need some terminology, in addition to that introduced in \cite{CF}.

Consider arbitrary sets $A,B,C$ and $D$, and let $f:A\rightarrow B$ and $g:C\rightarrow D$ be maps.   
The \emph{concatenation} of $g$ with $f$,
denoted $gf$, is defined to be the map with domain $f^{-1}[B\cap C]$ and codomain $D$ given by 
$(gf)(a)=g(f(a))$ for all $a\in f^{-1}[B\cap C]$. Note that if $B=C$, then the concatenation $gf$ is
simply the \emph{composition} of $g$ with $f$, i.e. $(gf)(a)=g(f(a))$ for all $a\in A$.
As in the particular case of composition, concatenation is associative. 

If $(g_i)_{i\in I}$ is a non-empty family of maps $g_i:A_i\rightarrow B_i$, 
where $(A_i)_{ i\in I}$ is a family of pairwise disjoint sets,
then their (\emph{piecewise}) \emph{sum}, denoted  ${\Sigma }_{i\in I}g_i$, is the map from 
${\cup }_{i\in I}A_i$ to ${\cup }_{i\in I}B_i$
 whose restriction to each $A_i$ agrees with $g_i$. 
 We also use $g_1+g_2$ to denote the 
sum of $g_1$ and $g_2$. In particular, if $B_1=B_2=B^n$, and $g_1$ and $g_2$ are the vector-valued functions
 $g_1=(g_{1}^1\ldots g_{1}^n)$ and 
 $g_2=(g_{2}^1\ldots g_{2}^n)$, where for each $1\leq j\leq n$, $g_{1}^j:A_1\rightarrow B$ and 
$g_{2}^j:A_2\rightarrow B$, then their sum is defined componentwise, i.e. $g_1+g_2$ is the 
vector-valued function defined by 
$(g_1+g_2)({\bf a})=((g_{1}^1+g_{2}^1)({\bf a})\ldots (g_{1}^n+g_{2}^n)({\bf a}))$,
 for every ${\bf a}\in A_1\cup A_2$.     
Clearly, piecewise sum is associative and commutative, and it is not difficult to see that
concatenation is distributive over sum.

Let $m$ be a positive integer (viewed as an ordinal), 
$(n_j)_{j\in J}$ be a non-empty family of positive integers (also viewed as ordinals),
 and let $V$ be an arbitrary set disjoint from $m$ and each member of $(n_j)_{j\in J}$. 
Any non-empty family $H=(h_j)_{j\in J}$ of maps $h_j:n_j\rightarrow m\cup V$,
 is called a \emph{minor formation scheme} with \emph{target} $m$,
 \emph{indeterminate set} $V$ and \emph{source family} $(n_j)_{j\in J}$.
If the indeterminate set $V$
is empty, i.e. for each $j\in J$, the maps $h_j$ have codomain $m$, then we say that the 
minor formation scheme $H=(h_j)_{j\in J}$ is \emph{simple}.

An $m$-ary $A$-to-$B$ constraint $(R,S)$ is said to be a \emph{conjunctive minor} of a non-empty family 
$(R_j,S_j)_{j\in J}$ of $A$-to-$B$ constraints (of various arities) \emph{via a scheme $H=(h_j)_{j\in J}$},
 if for every $m$-tuples ${\bf a}\in A^m$ and ${\bf b}\in B^m$, 
\begin{itemize}
\item[(a)] ${\bf a}\in R $ implies that there is a map $\sigma _A:V\rightarrow A$ such that,
 for all $j$ in $J$, we have $[({\bf a}+\sigma )h_j]\in R_j $, and
\item[(b)] if there is a map $\sigma _B:V\rightarrow B$ such that,
 for all $j$ in $J$, we have $[({\bf b}+\sigma )h_j]\in S_j $, then ${\bf b}\in S $.
\end{itemize}
The maps $\sigma _A$ and $\sigma _B$ are called \emph{Skolem maps}. If $(a)$ and $(b)$ hold
with "if and only if" replacing "implies" and "if", respectively, then $(R,S)$
 is called a \emph{tight conjunctive minor} of the family 
$(R_j,S_j)_{j\in J}$. (See \cite{CF} for further background.)
 If the minor formation scheme $H$ is simple,
then we say that $(R,S)$ is a \emph{weak conjunctive minor} of the family 
$(R_j,S_j)_{j\in J}$. Furthermore, if the scheme $H$ consists of identity maps on $m$, then $(R,S)$ is said
to be obtained by \emph{intersecting antecedents} and \emph{intersecting consequents} of the constraints in 
the family $(R_j,S_j)_{j\in J}$. In addition, if $J=\{0\}$, then conditions $(a)$ and $(b)$ above,
reduce to $R\subseteq R_0$ and $S\supseteq S_0$, respectively, and in this case
$(R,S)$ is called a \emph{relaxation} of $(R_0,S_0)$. We shall refer to relaxations $(R,S)$ with finite antecedent $R$
as \emph{finite relaxations}.

\begin{main1}
If $(R,S)$ is a conjunctive minor of a non-empty family $(R_j,S_j)_{j\in J}$ 
of $A$-to-$B$ constraints, and, for each $j\in J$, $(R_j,S_j)$ is a conjunctive minor of a non-empty family 
$(R_{j}^i,S_{j}^i)_{i\in I_j}$,
then $(R,S)$ is a conjunctive minor of the non-empty family $(R_{j}^i,S_{j}^i)_{j\in J,i\in I_j}$.
\end{main1}

\begin{proof}
The proof of the Transitivity Lemma follows as the proof of Claim 1 in \cite{CF} 
(see proof of Theorem 2), but the ordinals $m$ and $n_j$, for each $j\in J$, are assumed to be finite.

Suppose that $(R,S)$ is an $m$-ary conjunctive minor of the family $(R_j,S_j)_{j\in J}$ via 
a scheme $H=(h_j)_{j\in J}$, $h_j:n_j\rightarrow m\cup V$, and, for each $j\in J$,
 $(R_j,S_j)$ is an $n_j$-ary conjunctive minor of the family 
$(R_{j}^i,S_{j}^i)_{i\in I_j}$ via a scheme $H_j=(h_{j}^i)_{i\in I_j}$, 
$h_{j}^i:n_{j}^i\rightarrow n_j\cup V_j$, where the $V_j$'s are pairwise disjoint.
 
Consider the minor formation scheme $K=(k_{j}^i)_{j\in J, i\in I_j}$ defined as follows:
\begin{itemize}
\item[(i)] the target of $K$ is the target $m$ of $H$, 
\item[(ii)] the source family of $K$ is $(n_{j}^i)_{j\in J, i\in I_j}$, 
\item[(iii)] the indeterminate set of $K$ is $U=V\cup ({\cup }_{j\in J}V_j)$,
\item[(iv)] $k_{j}^i:n_{j}^i\rightarrow m\cup U$ is defined by
\begin{displaymath}
k_{j}^i=(h_j+\iota _{V_jU})h_{j}^i  
 \end{displaymath} 
\end{itemize}  where $\iota _{V_jU}$ is the canonical injection (inclusion map) on $V_j$ to $U$.
We show that $(R,S)$ is a conjunctive minor of the family $(R_{j}^i,S_{j}^i)_{j\in J,i\in I_j}$
 via the scheme $K=(k_{j}^i)_{j\in J, i\in I_j}$.

If ${\bf a}$ is an $m$-tuple in $R$, then there is a Skolem map $\sigma :V\rightarrow A$
 such that for all $j$ in $J$, $({\bf a}+\sigma )h_j\in R_j$.
Thus, for every $j$ in $J$, 
there are Skolem maps ${\sigma }_j:V_j\rightarrow A$ such that for every $i\in I_j$, we have
$[({\bf a}+\sigma )h_j+\sigma _j]h_{j}^i\in R_{j}^i$.

As in the proof of Claim 1 in \cite{CF}, let $\tau :U\rightarrow A$ be the Skolem map 
defined by $\tau =\sigma +{\Sigma }_{l\in J}{\sigma }_l$. By the fact that concatenation 
is associative and distributive over sum, 
it follows that for every $j\in J$ and $i\in I_j$, 
\begin{displaymath}
  ({\bf a}+\tau ){k_{j}^i}=({\bf a}+\sigma +{\Sigma }_{l\in J}{\sigma }_l)(h_j+\iota _{UV_j})h_{j}^i=
[({\bf a}+\sigma )h_j+\sigma _j]h_{j}^i .
 \end{displaymath} 
 Thus, for every $j\in J$ and $i\in I_j$, we have $({\bf a}+\tau ){k_{j}^i}\in R_{j}^i$. 

Now suppose that ${\bf b}$ is an $m$-tuple over $B$, for wich there is a Skolem map $\tau :U\rightarrow B$
such that for every $j\in J$ and $i\in I_j$, $({\bf b}+\tau ){k_{j}^i}$ is in $S_{j}^i$.
Consider the Skolem maps $\sigma :V\rightarrow B$ and 
${\sigma }_j:V_j\rightarrow B$ for every $j\in J$, such that $\tau =\sigma +{\Sigma }_{j\in J}{\sigma }_j$.
Again, by associativity and distributivity it follows that for every $j\in J$ and $i\in I_j$, 
$[({\bf b}+\sigma )h_j+\sigma _j]h_{j}^i=({\bf b}+\tau ){k_{j}^i}\in S_{j}^i$. 
Hence, for every $j\in J$, we have $({\bf b}+\sigma )h_j\in S_j$, and thus we conclude 
${\bf b}\in S$.
\end{proof}

Note that if $H$ is simple and, for every $j\in J$, the schemes $H_j$ are simple, then the scheme $K$ defined in the
proof above is also simple.
Thus we get:

\begin{main2}
If $(R,S)$ is a weak conjunctive minor of a non-empty family $(R_j,S_j)_{j\in J}$ 
of $A$-to-$B$ constraints, and, for each $j\in J$, $(R_j,S_j)$ is a weak conjunctive minor of a non-empty family 
$(R_{j}^i,S_{j}^i)_{i\in I_j}$,
then $(R,S)$ is a weak conjunctive minor of the non-empty family $(R_{j}^i,S_{j}^i)_{j\in J,i\in I_j}$.
\end{main2}

A set $\mathcal{T}$ of relational constraints is said to be 
\emph{closed under formation of} (\emph{weak}) \emph{conjunctive minors}
if whenever every member of a non-empty family $(R_j,S_j)_{j\in J}$ of constraints
 is in $\mathcal{T}$, then every (weak) conjunctive minor of the family $(R_j,S_j)_{j\in J}$ is
 also in $\mathcal{T}$. 
For any set of constraints $\mathcal{T}$, we denote by ${\bf CM}(\mathcal{T})$ 
the smallest set of constraints containing $\mathcal{T}$, 
and closed under formation of ``conjunctive minors". 
Similarly, we define ${\bf wCM}(\mathcal{T})$ to be
the smallest set of constraints containing $\mathcal{T}$, and 
closed under formation of weak conjunctive minors. 

By the Transitivity Lemma it follows that ${\bf CM}(\mathcal{T})$ is the set of all  
conjunctive minors of families of constraints in 
$\mathcal{T}$, and ${\bf CM}({\bf CM}(\mathcal{T}))={\bf CM}(\mathcal{T})$. Analogously, by 
the Transitivity Lemma for weak minors, it follows that ${\bf wCM}(\mathcal{T})$ is the set of all  
weak conjunctive minors of families of constraints in  
$\mathcal{T}$, and ${\bf wCM}({\bf wCM}(\mathcal{T}))={\bf wCM}(\mathcal{T})$. In other words, both 
$\mathcal{T}\mapsto {\bf CM}(\mathcal{T})$ and 
$\mathcal{T}\mapsto {\bf wCM}(\mathcal{T})$ are idempotent maps.
Furthermore, both $\mathcal{T}\mapsto {\bf CM}(\mathcal{T})$ and 
$\mathcal{T}\mapsto {\bf wCM}(\mathcal{T})$ are monotone and extensive (in the sense of Section 2),
 and hence, we have:

 \begin{fact} 
The operators $\mathcal{T}\mapsto {\bf CM}(\mathcal{T})$ and 
$\mathcal{T}\mapsto {\bf wCM}(\mathcal{T})$ are closure operators on the set of all $A$-to-$B$ constraints. 
\end{fact}

The following technical result shows that the sets of constraints characterized by multivalued functions, and by total
multivalued functions must be closed under formation of weak conjunctive minors,
 and closed under formation of conjunctive minors, respectively.

\begin{lemma}
Let $(R_j,S_j)_{j\in J}$ be a non-empty family of $A$-to-$B$ constraints.
If $f:A^n\rightarrow \mathcal{P}(B)$ satisfies every $(R_j,S_j)$ then $f$ satisfies every 
weak conjunctive minor of the family $(R_j,S_j)_{j\in J}$. Futhermore, if $f$ is total, then
$f$ satisfies every conjunctive minor of the family $(R_j,S_j)_{j\in J}$. 
\end{lemma}

\begin{proof}
First we prove the last claim, which generalizes Lemma 1 in \cite{CF}, 
to total multivalued functions.
Let $f$ be a total multivalued function, say of arity $n$, satisfying
 every member of a non-empty family $(R_j,S_j)_{j\in J}$ of $A$-to-$B$ constraints, and let
 $(R,S)$ be an $m$-ary conjunctive minor of the family $(R_j,S_j)_{j\in J}$ via a scheme  
$H=(h_j)_{j\in J}$, $h_j:n_j\rightarrow m\cup V$.
We show that for every ${\bf a}^1\ldots {\bf a}^n\in R$, 
 the $m$-ary relation $f({\bf a}^1\ldots {\bf a}^n)$ is contained in $S$, i.e. $f$ satisfies $(R,S)$.
So let ${\bf a}^1\ldots {\bf a}^n$ be any $m$-tuples in $R$. 
Observe that for each $1\leq i\leq n$, 
there is a Skolem map ${\sigma }_i:V\rightarrow A$, such that for every $j$ in $J$,
$({\bf a}^i+\sigma _i)h_j$ is in $R_j$. 
Since $f$ satisfies every member of $(R_j,S_j)_{j\in J}$, we have that for every $j$ in $J$, 
$f[({\bf a}^1+{\sigma }_1)h_j\ldots ({\bf a}^n+{\sigma }_n)h_j]\subseteq S_j$.

Now, suppose that ${\bf b}\in f({\bf a}^1\ldots {\bf a}^n)$.
Since $f$ is a total multivalued function, there is a Skolem map 
$\sigma :V\rightarrow B$ such that, for every $v\in V$,
 $\sigma (v)$ belongs to $f({\sigma }_1(v)\ldots {\sigma }_n(v))$. 
Fix such a Skolem map $\sigma :V\rightarrow B$.
 By associativity and distributivity of concatenation over sum,
we have that for each $j$ in $J$,
\begin{displaymath}
({\bf b}+\sigma )h_j\in [f({\bf a}^1\ldots {\bf a}^n)+f({\sigma }_1\ldots {\sigma }_n)]h_j
=f[({\bf a}^1+{\sigma }_1)h_j\ldots ({\bf a}^n+{\sigma }_n)h_j]\subseteq S_j.
\end{displaymath}
Since $(R,S)$ is a conjunctive minor of $(R_j,S_j)_{j\in J}$
 via the scheme $H=(h_j)_{j\in J}$, we conclude that ${\bf b}\in S$, 
 which completes the proof of the last statement of Lemma 1.

To prove the first claim of Lemma 1, suppose that $f:A^n\rightarrow \mathcal{P}(B)$ is a multivalued
function, not necessarily total, satisfying
 every member of $(R_j,S_j)_{j\in J}$, and assume that
 $(R,S)$ is a weak conjunctive minor of the family $(R_j,S_j)_{j\in J}$, say via a scheme  
$H=(h_j)_{j\in J}$, where $h_j:n_j\rightarrow m$ for every $j$ in $J$.

As before, we prove that $f$ satisfies $(R,S)$, by showing that for every ${\bf a}^1\ldots {\bf a}^n$ in $R$, 
we have $f({\bf a}^1\ldots {\bf a}^n)\subseteq S$. 
Clearly, if $f(({\bf a}^1\ldots {\bf a}^n)(i))=\emptyset $, for some $i\in m$,  
 then $f({\bf a}^1\ldots {\bf a}^n)\subseteq S$. 
So we may assume that $f(({\bf a}^1\ldots {\bf a}^n)(i))\not=\emptyset $, for every $i\in m$.
As before, for each $j$ in $J$, the $n_j$-tuples
 ${\bf a}^1h_j\ldots {\bf a}^nh_j$ belong to $R_j$, and
since $f$ satisfies each $(R_j,S_j)$, we have that $f({\bf a}^1h_j\ldots {\bf a}^nh_j)\subseteq  S_j$, for every $j\in J$.
By associativity, it follows that for each $j$ in $J$,
\begin{displaymath}
[f({\bf a}^1\ldots {\bf a}^n)]h_j
=f[({\bf a}^1\ldots {\bf a}^n)h_j]=f({\bf a}^1h_j\ldots {\bf a}^nh_j)\subseteq S_j.
\end{displaymath}
Therefore, if ${\bf b}\in f({\bf a}^1\ldots {\bf a}^n)$, then for every $j\in J$,
we have ${\bf b}h_j\in S_j$, which implies ${\bf b}\in S$,
 because $(R,S)$ is a weak conjunctive minor of $(R_j,S_j)_{j\in J}$
 via the scheme $H=(h_j)_{j\in J}$. Thus $f({\bf a}^1\ldots {\bf a}^n)$ is indeed contained in $S$, and
  the proof of Lemma 1 is complete.
\end{proof}

In order to describe the Galois closed sets of constraints,
we need to recall a further condition, introduced in \cite{CF}, which expresses "compactness" on the sets of 
these dual objects.
A set $\mathcal{T}$ of relational constraints is said to be \emph{locally closed} if $\mathcal{T}$ 
contains every constraint $(R,S)$ such that the set of all its finite relaxations, is contained in $\mathcal{T}$.
In analogy with Section 3, we denote by ${\bf LO}(\mathcal{T})$, the smallest locally closed
set of constraints containing $\mathcal{T}$.
Similarly to the closure $\bf LC$ defined on classes of function classes, 
${\bf LO}(\mathcal{T})$ is the set of
constraints obtained from $\mathcal{T}$ by adding all those constraints whose finite relaxations
 are all in $\mathcal{T}$. As an immediate consequence, we have:

 \begin{fact} 
The operator $\mathcal{T}\mapsto {\bf LO}(\mathcal{T})$ is a closure operator
 on the set of all $A$-to-$B$ constraints. 
\end{fact}

Note that in the case of finite underlying sets $A$ and $B$, the induced operator in Fact 3 is the identity
map, i.e. every set of constraints is locally closed.

\begin{theorem}
Consider arbitrary non-empty sets $A$ and $B$.
Let $\mathcal{T}$ be a set of $A$-to-$B$ relational constraints. Then the following are equivalent:
\begin{itemize}
\item[(i)]$\mathcal{T}$ is locally closed, contains the unary empty constraint $(\emptyset ,\emptyset )$ 
and the unary trivial constraint $(A,B)$, 
 and is closed under formation of weak conjunctive minors;
\item[(ii)]$\mathcal{T}$ is characterized by some set of multivalued functions on $A$ to $B$.
\end{itemize}
\end{theorem}

\begin{proof}
To see that $(ii)$ implies $(i)$, note first that every 
multivalued function satisfies the empty constraint $(\emptyset ,\emptyset )$,
 and the trivial constraint $(A,B)$.
 Also, from Lemma 1, it follows that every set of constraints characterized
by multivalued functions is closed under formation of weak conjunctive minors.
For the remainder, let $(R,S)$ be any constraint not in $\mathcal{T}$.  
 By $(ii)$ it follows that there is an $n$-ary multivalued function $f$ satisfying every constraint in 
$\mathcal{T}$ which does not satisfy $(R,S)$, i.e.  
 there are ${\bf a}^1\ldots {\bf a}^n\in R$, 
such that $f({\bf a}^1\ldots {\bf a}^n)\not\subseteq S$. 
Let $F$ be the subset of $R$ containing ${\bf a}^1\ldots {\bf a}^n$.
 Clearly, the constraint $(F,S)$ is a finite relaxation of 
$(R,S)$, and $(F,S)\not\in \mathcal{T}$. 
Since the above argument works for any constraint not in $\mathcal{T}$, we conclude that
 $\mathcal{T}$ is locally closed.

To prove implication $(i)\Rightarrow (ii)$, let 
$(R,S)$ be any constraint not in ${\mathcal{T}}$, say of arity $m$.
We show that there is a multivalued function separating $(R,S)$ from ${\mathcal{T}}$.
From the fact that $\mathcal{T}$ is locally closed, it follows that
there is a finite relaxation $(F,S_0)$ of $(R,S)$, say with $F$ of size $n$,
 which is not in $\mathcal{T}$. 
Observe that $(F,B^m)$ is a weak conjunctive minor of the 
 unary trivial constraint $(A,B)$, and so we must have $S_0\not=B^m$.
Also, $F$ can not be empty because $(\emptyset ^m,S_0)$ is a relaxation of the 
$m$-ary empty constraint, which in turn is a weak conjunctive minor of $(\emptyset ,\emptyset )$. 
From the fact that $(F,S_0)$ can be obtained from the family 
$(F,B^m\setminus \{{\bf s}\})_{{\bf s}\not\in S_0}$, by intersecting antecedents and intersecting consequents,
 it follows that there must exist an $m$-tuple ${\bf s}=(s_i\mid i\in m)$ in $B^m$ which
is not in $S_0$, and such that $(F,B^m\setminus \{{\bf s}\})$ 
does not belong to $\mathcal{T}$. 
Let ${\bf a}^1,\ldots ,{\bf a}^n$ be the $m$-tuples in $F$.

 We define a multivalued function which is not empty-valued on 
 \begin{displaymath}
D=\{({\bf a}^1\ldots {\bf a}^n)(i):i\in m\}
\end{displaymath} 
 but empty-valued on the remaining $n$-tuples of $A^n$.
 Formally, let $g$ be the $n$-ary multivalued function 
on $A$ to $B$ such that, for every $i\in m$, 
 \begin{displaymath}
g(({\bf a}^1\ldots {\bf a}^n)(i))=
\cup \{s_j:j\in m \textrm{ and } ({\bf a}^1\ldots {\bf a}^n)(j)=({\bf a}^1\ldots {\bf a}^n)(i)\}, 
\end{displaymath} 
and $g({\bf a})=\emptyset $ for every ${\bf a}\in A^n\setminus D$.
 From definition of $g$, it follows that ${\bf s}\in g({\bf a}^1\ldots {\bf a}^n)$,
 and thus $g$ does not satisfy $(F,S_0)$,
 and so it does not satisfy $(R,S)$.

Now we show that $g$ satisfies every member of ${\mathcal{T}}$.
For a contradiction, suppose that there is an $m_1$-ary member $(R_1,S_1)$ of ${\mathcal{T}}$,
 which is not satisfied by $g$.
Thus, for some ${\bf c}^1,\ldots ,{\bf c}^n \in R_1$ we have 
$g({\bf c}^1\ldots {\bf c}^n) \not\subseteq S_1$. 
Consider an $m_1$-tuple ${\bf s}_1\in g({\bf c}^1\ldots {\bf c}^n)\setminus  S_1$, and 
let $h:m_1\rightarrow m$ to be any map such that 
 \begin{displaymath}
{\bf s}_1(i)=({\bf s}h)(i).
\end{displaymath} 
 for every $i\in m_1$.
Note that for every $i\in m_1$, there is $j\in m$ such that 
$({\bf c}^1\ldots {\bf c}^n)(i)=({\bf a}^1\ldots {\bf a}^n)(j)$, for otherwise 
$g({\bf c}^1\ldots {\bf c}^n)$ would be empty
and so would be contained in $S_1$. In fact, from definition of $g$ and $h$, 
it follows that, for every $i\in m_1$,
 $({\bf c}^1\ldots {\bf c}^n)(i)=({\bf a}^1h\ldots {\bf a}^nh)(i)$.

Let $(R_h,S_h)$ be the $m$-ary weak conjunctive minor of $(R_1,S_1)$ via $H=\{h\}$,
 defined by 
\begin{itemize}
\item[(a)] for every ${\bf a}\in A^m$, ${\bf a}\in R_h$ if and only if ${\bf a}h\in R_1$ , and
\item[(b)] for every ${\bf b}\in B^m$, ${\bf b}\in S_h$ if and only if ${\bf b}h\in S_1$.
\end{itemize}
Observe that $(R_h,S_h)$ belongs to ${\mathcal{T}}$, because 
${\mathcal{T}}$ is closed under formation of weak conjunctive minors.

Since ${\bf a}^1h,\ldots ,{\bf a}^nh\in R_1$, we have ${\bf a}^1,\ldots ,{\bf a}^n\in R_h$, i.e. 
$F\subseteq R_h$. Also, ${\bf s}\not\in S_h$ because ${\bf s}_1\not\in S_1$.
 Therefore $(F,B^m\setminus \{{\bf s}\})$
 is a relaxation of $(R_h,S_h)$, and we conclude that 
$(F,B^{m}\setminus \{{\bf s}\})\in {\mathcal{T}}$, which is a contradiction.
Thus $g$ is indeed a multivalued function separating $(R,S)$ from  ${\mathcal{T}}$.
\end{proof}

In Section 2, we observed that every partial function satisfies the binary equality constraint, thus
any set of constraints characterized by partial functions must contain this constraint. 
In fact, this additional condition is also sufficient to describe the Galois closed sets of constraints
associated with the correspondence ${\bf pFSC}-{\bf CSF}$:

\begin{cor}
Consider arbitrary non-empty sets $A$ and $B$.
Let $\mathcal{T}$ be a set of $A$-to-$B$ relational constraints. Then the following are equivalent:
\begin{itemize}
\item[(i)]$\mathcal{T}$ is locally closed, contains the unary empty constraint, 
the unary trivial constraint and the binary equality constraint, 
 and is closed under formation of weak conjunctive minors;
\item[(ii)]$\mathcal{T}$ is characterized by some set of partial functions on $A$ to $B$.
\end{itemize}
\end{cor}

\begin{proof}
The implication $(ii)\Rightarrow (i)$, is a consequence of Theorem 2 and the observations above. 
The proof of the implication $(i)\Rightarrow (ii)$, follows exactly as the proof of $(ii)\Rightarrow (i)$ 
in Theorem 2. The key observation is that if $\mathcal{T}$ contains the binary equality 
constraint $(=_A,=_B)$, and it is closed under formation of weak conjunctive minors, then 
for every $i,j\in m$ such that $i\not=j$, the $m$-ary $A$-to-$B$ constraint 
$(R_{ij},S_{ij})$ defined by 
 \begin{displaymath}
 R_{ij}=\{(a_t\mid t\in m ): a_i=a_j \} \quad \textrm{ and } \quad 
 S_{ij}=\{(b_t\mid t\in m ): b_i=b_j \}
\end{displaymath} 
is in $\mathcal{T}$. From this fact, we have that
 in the proof of $(ii)\Rightarrow (i)$ of Theorem 2, the following holds for every $i,j\in m$:
 
\emph{if $({\bf a}^1\ldots {\bf a}^n)(j)=({\bf a}^1\ldots {\bf a}^n)(i)$, 
then ${\bf s}(j)={\bf s}(i)$}.

Indeed, as observed in the proof of Theorem 2 in \cite{CF}, if for some $i\not=j$, we have 
$({\bf a}^1\ldots {\bf a}^n)(j)=({\bf a}^1\ldots {\bf a}^n)(i)$, 
but ${\bf s}(j)\not={\bf s}(i)$, then $(F,B^m\setminus \{{\bf s}\})$ would be a relaxation
of $(R_{ij},S_{ij})$, and hence would be in $\mathcal{T}$, which is a contradiction.

Thus the separating function $g$ defined in the proof of $(i)\Rightarrow (ii)$ of Theorem 2,
 is in fact a partial function,
which completes the proof of Theorem 3.  
\end{proof}

The next two results are the analogues of Theorem 2 and Corollary 2, which show that, in addition, 
closure under formation of conjunctive minors suffices to describe 
the sets of relational constraints characterized by total functions.

\begin{theorem}
Consider arbitrary non-empty sets $A$ and $B$.
Let $\mathcal{T}$ be a set of $A$-to-$B$ relational constraints. Then the following are equivalent:
\begin{itemize}
\item[(i)]$\mathcal{T}$ is locally closed, contains the unary empty constraint
and the unary trivial constraint, 
 and is closed under formation of conjunctive minors;
\item[(ii)]$\mathcal{T}$ is characterized by some set of total multivalued functions on $A$ to $B$.
\end{itemize}
\end{theorem}

\begin{proof}
The proof of implication $(ii)\Rightarrow (i)$ follows as the
 proof of Theorem 2 (using the last statement in Lemma 1).
 To prove $(i)\Rightarrow (ii)$ we shall make use of notions and terminology, as well as few results
 particular to the proof of Theorem 2 in [CF]. Ordinals are allowed to be infinite,
unless they denote function arities which remain finite. 
Thus the relations and constraints considered in
this proof may be infinitary. Also, in minor formation schemes, 
the targets and members of the source families
are allowed to be arbitrary, possibly infinite, non-zero ordinals, so that the notion of
 conjunctive minor is naturally extended to this more general setting.
 We shall use the term "conjunctive
$\infty $-minor" to indicate a conjunctive minor which may be finitary or infinitary. 
As shown in \cite{CF} (see Claim 1 in the proof of Theorem 2),
the Transitivity Lemma is extended to this general setting: 

\begin{main3}\emph{(\cite{CF}:)}
If $(R,S)$ is a conjunctive $\infty $-minor of a non-empty family $(R_j,S_j)_{j\in J}$ 
of $A$-to-$B$ constraints, and, for each $j\in J$, $(R_j,S_j)$ is a conjunctive $\infty $-minor of a non-empty family 
$(R_{j}^i,S_{j}^i)_{i\in I_j}$,
then $(R,S)$ is a conjunctive $\infty $-minor of the non-empty family $(R_{j}^i,S_{j}^i)_{j\in J,i\in I_j}$.
\end{main3}
 
A proof of the Infinitary Transitivity can be obtained by allowing infinite ordinals, in
 the proof of the Transitivity Lemma.
 We use Infinitary Transitivity to prove the analogue of Claim 2 in the proof of Theorem 2 in \cite{CF}:

 \begin{main}
Let $\mathcal{T}$ be a locally closed set of finitary $A$-to-$B$ constraints containing 
the unary empty constraint and the unary trivial constraint, 
 and closed under formation of conjunctive minors, and let ${\mathcal{T}}^{\infty }$ 
be its closure under formation of conjunctive $\infty $-minors. 
Let $(R,S)$ be a finitary $A$-to-$B$ constraint not in $\mathcal{T}$.
Then there is a total multivalued function $g$ on $A$ to $B$ such that   
\begin{itemize}
\item[(1)] $g$ satisfies every constraint in ${\mathcal{T}}^{\infty }$, and 
\item[(2)] $g$ does not satisfy $(R,S)$.
\end{itemize}
\end{main}

Observe that by the Infinitary Transitivity, $\mathcal{T}$ is the set of all finitary constraints
 in ${\mathcal{T}}^{\infty }$.

{\sl Proof of Claim. }
Proceeding in analogy with the proof of Theorem 2,
we construct a total multivalued function $g$ which satisfies all constraints in 
${\mathcal{T}}^{\infty }$ but $g$ does not satisfy $(R,S)$.

Let $m$ be the arity of $(R,S)\not\in \mathcal{T}$. By the comment following the Claim,
 $(R,S)$ can not be in ${\mathcal{T}}^{\infty }$.
As in the proof of $(i)\Rightarrow (ii)$ in Theorem 2, let $(F,S_0)$ be a relaxation of $(R,S)$
 with finite antecedent, not in $\mathcal{T}$. As before, $F$ cannot be empty, and $S_0\not=B^m$.   
Let $F=\{{\bf d}^1,\ldots ,{\bf d}^n\}$ of finite size $1\leq n$. 

 Let $\mu =\mid A^n\mid $, and consider $\mu $-tuples ${\bf a}^1\ldots {\bf a}^n\in A^\mu $ such that 
 \begin{displaymath}
({\bf a}^1\ldots {\bf a}^n)(i)=({\bf d}^1\ldots {\bf d}^n)(i), \textrm{ for every $i\in m$} 
\end{displaymath} 
 and such that $\{({\bf a}^1\ldots {\bf a}^n(i):i\in \mu \backslash m\}$ are 
 the remaining distinct $n$-tuples in $A^n$ without repetitions.
Let $R_F$ be the $\mu $-ary relation defined by $R_F=\{{\bf a}^1,\ldots ,{\bf a}^n\}$, and let $S_F$
 be the $\mu $-ary relation comprising all those $\mu $-tuples ${\bf b}=(b_t\mid t\in \mu )$ in $B^{\mu }$
 such that $(b_t\mid t\in m)$ belongs to $S_0$. 
By the Infinitary Transitivity and the comment following Claim,
 it follows that $(R_F,S_F)\not\in {\mathcal{T}}^{\infty }$.

In analogy with the proof of Theorem 2, let ${\bf s} = (s_t\mid t\in \mu )$ be a $\mu $-tuple in $B^{\mu }$ such that 
$(s_t\mid t\in m)$ is not in $S_0$, and for which $(R_F,B^{\mu }\setminus \{{\bf s}\})$ 
does not belong to $\mathcal{T}^{\infty }$.
Consider the $n$-ary multivalued function $g$ on $A$ to $B$, defined by
 \begin{displaymath}
g(({\bf a}^1\ldots {\bf a}^n)(i))=
\cup \{s_j: j\in \mu \textrm{ and }({\bf a}^1\ldots {\bf a}^n)(j)=({\bf a}^1\ldots {\bf a}^n)(i)\},
\end{displaymath}  for every $i\in \mu $. 
Note that for every ${\bf a}\in \{({\bf a}^1\ldots {\bf a}^n)(i):i\in \mu \}=A^n$, we have
$g({\bf a})\not=\emptyset $, that is, $g$ is total.
Also, ${\bf s}\in g({\bf a}^1\ldots {\bf a}^n)$,
 thus $g$ does not satisfy $(R_F,S_F)$. Since $(R_F,S_F)$ is a conjunctive minor of $(F,S_0)$, it follows 
from Lemma 1 that $g$ does not satisfy $(F,S_0)$
  and hence, it does not satisfy $(R,S)$. 
   
Now we show that $g$ also satisfies $(1)$. For a contradiction, 
suppose that there is a $\rho $-ary constraint $(R_1,S_1)\in {\mathcal{T}}^{\infty }$,
 which is not satisfied by $g$. That is,
 for some ${\bf c}^1,\ldots ,{\bf c}^n$ in $R_1$ we have 
$g({\bf c}^1\ldots {\bf c}^n)\not\subseteq  S_1$.
Let ${\bf s}_1$ be an $\rho $-tuple in $g({\bf c}^1\ldots {\bf c}^n)$
such that ${\bf s}_1\not\in S_1$, and  
let $h:\rho \rightarrow \mu $ be any map such that, for every $i\in \rho $:
\begin{itemize}
\item[(a)] ${\bf s}_1(i)=({\bf s}h)(i) $, and 
\item[(b)] $({\bf c}^1\ldots {\bf c}^n)(i)=({\bf a}^1h\ldots {\bf a}^nh)(i)$.
\end{itemize}
Note that $(b)$ implies that ${\bf c}^j={\bf a}^jh$, for every $1\leq j\leq n$.   

Let $(R_h,S_h)$ be the $\mu $-ary tight conjunctive $\infty $-minor of $(R_1,S_1)$ via $H=\{h\}$,
 i.e. for every $\mu $-tuple ${\bf a}$ of $A^{\mu }$, ${\bf a}\in R_h$ if and only if ${\bf a}h\in R_1$,
 and for every $\mu $-tuple ${\bf b}$ of $B^{\mu }$, ${\bf b}\in S_h$ if and only if  
${\bf b}h\in S_1$. Clearly, ${\bf a}^1,\ldots ,{\bf a}^n\in  R_h$, that is, $R_F\subseteq R_h$, and
 ${\bf s}\not\in S_h$. Thus $(R_F,B^{\mu }\setminus \{{\bf s}\})$
 is a relaxation of $(R_h,S_h)$, and, since ${\mathcal{T}}^{\infty }$ is closed under 
formation of conjunctive $\infty $-minors, it follows from the Infinitary Transitivity that  
$(R_F,B^{\mu }\setminus \{{\bf s}\})\in {\mathcal{T}}^{\infty }$, yielding the desired contradiction, and
the proof of the Claim is complete.
 
\smallskip

By the Claim above, it follows that for every constraint $(R,S)$ not in $\mathcal{T}$ there is a total
multivalued function $g$ on $A$ to $B$ 
which does not satisfy $(R,S)$ but satisfies in particular every constraint in $\mathcal{T}$.
In other words, the implication $(i)\Rightarrow (ii)$ also holds.
\end{proof}

Note that the unary trivial constraint $(A,B)$, is a tight conjunctive minor of the binary equality 
constraint $(=_A,=_B)$.

\begin{cor} \emph{(In \cite{CF}:)}
Consider arbitrary non-empty sets $A$ and $B$.
Let $\mathcal{T}$ be a set of $A$-to-$B$ relational constraints. Then the following are equivalent:
\begin{itemize}
\item[(i)]$\mathcal{T}$ is locally closed, contains the unary empty constraint
and the binary equality constraint, and it is closed under formation of conjunctive minors;
\item[(ii)]$\mathcal{T}$ is characterized by some set of single-valued functions on $A$ to $B$.
\end{itemize}
\end{cor}

\begin{proof}
The implication $(i)\Rightarrow (ii)$, is a consequence of Corollary 2 and Theorem 4. 
The proof of $(ii)\Rightarrow (i)$ is analogous to that of Corollary 2,
but following the lines in the proof of $(ii)\Rightarrow (i)$ of Theorem 4. 
\end{proof}

In order to factorize the closure operators associated with the Galois connections 
for generalized functions and constraints defined in section 2, as compositions of the operators 
${\bf LO}$, ${\bf wCM}$, and ${\bf CM}$, we shall make use of the following analogues of $(iii)$ and 
$(iv)$ in Proposition 1:

\begin{proposition} 
Consider arbitrary non-empty sets $A$ and $B$, and let $\mathcal{T}$ be a set of $A$-to-$B$ relational
constraints.
\begin{itemize}
\item[(i)] If ${\bf CM}(\mathcal{T})=\mathcal{T}$,
 then ${\bf CM}({\bf LO}(\mathcal{T}))={\bf LO}(\mathcal{T})$.
\item[(ii)] If ${\bf wCM}(\mathcal{T})=\mathcal{T}$,
 then ${\bf wCM}({\bf LO}(\mathcal{T}))={\bf LO}(\mathcal{T})$.
\end{itemize}
\end{proposition}

\begin{proof}
We follow the strategy used in the proof of Proposition 1 (b) in \cite{Co}. 
 By Fact 2, to prove $(i)$ we only need to show that 
${\bf CM}({\bf LO}(\mathcal{T}))\subseteq {\bf LO}(\mathcal{T})$, i.e. that every 
conjunctive minor of a family of constraints in ${\bf LO}(\mathcal{T})$,
 is also in ${\bf LO}(\mathcal{T})$. 
So let $(R,S)$ be a conjunctive minor of a non-empty family $(R_j,S_j)_{j\in J}$
 of constraints in ${\bf LO}(\mathcal{T})$ via a scheme $H=(h_j)_{j\in J}$ with indeterminate set $V$.
 Consider the tight conjunctive minor $(R_0,S_0)$ of the family $(R_j,S_j)_{j\in J}$ 
via the same scheme $H$.
Since every relaxation of $(R,S)$ is a relaxation of $(R_0,S_0)$, in order to prove that
$(R,S)\in {\bf LO}(\mathcal{T})$, it is enough to show that every finite relaxation of
 $(R_0,S_0)$ is in $\mathcal{T}$.

Let $(F,S')$ be a finite relaxation of $(R_0,S_0)$,
 say $F$ having $n$ distinct elements ${\bf a}_1, \ldots ,{\bf a}_n$.
Note that for every ${\bf a}_i\in F$,
there is a Skolem map ${\sigma }_i:V\rightarrow A$ such that, for all $j$ in $J$, we have
$({\bf a}_i+{\sigma }_i)h_j\in R_j$.
For each $j$ in $J$, let $F_j$ be the subset of $R_j$, given by  
\begin{displaymath}
F_j=\{({\bf a}_i+{\sigma }_i)h_j: {\bf a}_i\in F\}. 
\end{displaymath}
Clearly, $(F,S')$ is a conjunctive minor of the family $(F_j,S_j)_{j\in J}$, and for each
$j$ in $J$, $(F_j,{S}_j)$ is a finite relaxation of $(R_j,S_j)$. 
Since ${\bf CM}(\mathcal{T})=\mathcal{T}$, and for each $j$ in $J$,
$(R_j,S_j)$ is in ${\bf LO}(\mathcal{T})$,   
we have that every member of the family $(F_j,S_j)_{j\in J}$ belongs to $\mathcal{T}$.  
 Hence $(F,S')$ 
is a conjunctive minor of a family of members of $\mathcal{T}$, and thus $(F,S')$  
is also in $\mathcal{T}$. 

The proof of $(ii)$ can be easily obtained by substituting "conjunctive minor" for "weak conjunctive minor",
and defining the finite subsets $F_j$ of $R_j$, by
$F_j=\{{\bf a}_ih_j: {\bf a}_i\in F\}$.  
\end{proof}

In other words, ${\bf LO}({\bf wCM}(\mathcal{T}))$ and ${\bf LO}({\bf CM}(\mathcal{T}))$ are the smallest
locally closed sets of constraints containing $\mathcal{T}$, which are 
closed under formation of weak conjunctive minors and 
closed under formation of conjunctive minors, respectively.
Using the characterizations of the Galois closed sets of constraints, we obtain the following decompositions 
of the closure operators associated with the corresponding Galois connections:

\begin{proposition}
Consider arbitrary non-empty sets $A$ and $B$.
 For any set $\mathcal{T}$ of $A$-to-$B$ relational constraints, the following hold:
\begin{itemize}
\item[(i)] ${\bf CSF}({\bf mFSC}(\mathcal{T}))=
{\bf LO}({\bf wCM}(\mathcal{T}\cup \{(\emptyset ,\emptyset ),(A,B)\}))$,
\item[(ii)] ${\bf CSF}({\bf pFSC}(\mathcal{T}))=
{\bf LO}({\bf wCM}(\mathcal{T}\cup \{(\emptyset ,\emptyset ),(A,B),(=_A,=_B)\}))$, 
\item[(iii)] ${\bf CSF}({\bf tFSC}(\mathcal{T}))=
{\bf LO}({\bf CM}(\mathcal{T}\cup \{(\emptyset ,\emptyset ),(A,B)\}))$, and
\item[(iv)] ${\bf CSF}({\bf sFSC}(\mathcal{T}))=
{\bf LO}({\bf CM}(\mathcal{T}\cup \{(\emptyset ,\emptyset ),(=_A,=_B)\}))$.
\end{itemize}
\end{proposition}

\end{document}